\documentclass[a4paper,10pt]{amsart}

\usepackage{amsmath}
\usepackage{amsthm}
\usepackage{amssymb}
\usepackage{amsfonts}
\usepackage{enumitem} 
\usepackage[bookmarks=true,hyperindex,pdftex,colorlinks,citecolor=blue]{hyperref}

\DeclareMathOperator{\lspan}{span}                          
\DeclareMathOperator{\diam}{diam}                           

\DeclareMathOperator{\Lip}{Lip}                             

\newcommand{\NN}{\mathbb{N}}                                
\newcommand{\RR}{\mathbb{R}}                                
\newcommand{\SA}{\operatorname{SNA}}
\newcommand{\NA}{\operatorname{NA}}

\newcommand{\abs}[1]{\left|{#1}\right|}                     
\newcommand{\pare}[1]{\left({#1}\right)}                    
\newcommand{\set}[1]{\left\{{#1}\right\}}                   
\newcommand{\norm}[1]{\left\|{#1}\right\|}                  
\newcommand{\sphere}[1]{S_{{#1}}}                           
\newcommand{\duality}[1]{\left<{#1}\right>}                 
\newcommand{\restrict}{\mathord{\upharpoonright}}           

\newcommand{\lipfree}[1]{\mathcal{F}({#1})}                 
\newcommand{\lipnorm}[1]{\norm{#1}_L}                       
\newcommand{\mol}[1]{m_{#1}}                                

\newcommand{\pten}{\,\ensuremath{\widehat{\otimes}_\pi}\,}

\theoremstyle{plain}
\newtheorem{theorem}{Theorem}[section]
\newtheorem{lemma}[theorem]{Lemma}
\newtheorem{corollary}[theorem]{Corollary}
\newtheorem{proposition}[theorem]{Proposition}

\theoremstyle{definition}
\newtheorem*{definition*}{Definition}

\newtheorem{example}[theorem]{Example}
\newtheorem{question}{Question}

\theoremstyle{remark}
\newtheorem{remark}[theorem]{Remark}

\begin{document}

\title{Points of differentiability of the norm in Lipschitz-free spaces}

\author[R. J. Aliaga]{Ram\'on J. Aliaga}
\address[R. J. Aliaga]{Instituto Universitario de Matem\'atica Pura y Aplicada,
Universitat Polit\`ecnica de Val\`encia,
Camino de Vera S/N,
46022 Valencia, Spain}
\email{raalva@upvnet.upv.es}
\thanks{R. J. Aliaga was partially supported by the Spanish Ministry of Economy, Industry and Competitiveness under Grant MTM2017-83262-C2-2-P, and by a travel grant of the Institute of Mathematics (IEMath-GR) of the University of Granada, Spain.}

\author[A. Rueda Zoca]{Abraham Rueda Zoca}\thanks{The research of Abraham Rueda Zoca was supported by Vicerrectorado de Investigaci\'on y Transferencia de la Universidad de Granada in the program ``Contratos puente'', by MICINN (Spain) Grant PGC2018-093794-B-I00 (MCIU, AEI, FEDER, UE), by Junta de Andaluc\'ia Grant A-FQM-484-UGR18
and by Junta de Andaluc\'ia Grant FQM-0185}
\address[A. Rueda Zoca]{Universidad de Granada,
Facultad de Ciencias,
Departamento de An\'alisis Matem\'atico,
18071 Granada, Spain}
\email{abrahamrueda@ugr.es}
\urladdr{\url{https://arzenglish.wordpress.com}}

\date{} 


\begin{abstract}
We consider convex series of molecules in Lipschitz-free spaces, i.e. elements of the form $\mu=\sum_n \lambda_n \frac{\delta_{x_n}-\delta_{y_n}}{d(x_n,y_n)}$ such that $\|\mu\|=\sum_n |\lambda_n |$. We characterise these elements in terms of geometric conditions on the points $x_n$, $y_n$ of the underlying metric space, and determine when they are points of G\^ateaux differentiability of the norm. In particular, we show that G\^ateaux and Fr\'echet differentiability are equivalent for finitely supported elements of Lipschitz-free spaces over uniformly discrete and bounded metric spaces, and that their tensor products with G\^ateaux (resp. Fr\'echet) differentiable elements of a Banach space are G\^ateaux (resp. Fr\'echet) differentiable in the corresponding projective tensor product.
\end{abstract}




\maketitle


\section{Introduction}

A \emph{pointed metric space} is just a metric space $M$ in which we distinguish an element, called $0$. Given a pointed metric space $M$ and a Banach space $X$, we write $\Lip(M,X)$ ($\Lip(M)$ when $X=\mathbb R$) to denote the Banach space of all Lipschitz maps $f: M\longrightarrow X$ which vanish at $0$, endowed with the Lipschitz norm defined by
$$ \Vert f \Vert_L := \sup\left\{\frac{\Vert f(x)-f(y)\Vert }{d(x,y)} \colon x,y\in M,\, x \neq y \right\}.$$
We denote by $\delta$ the canonical isometric embedding of $M$ into $\Lip(M)^*$, which is given by $\langle f, \delta(x) \rangle =f(x)$ for $x \in M$ and $f \in \Lip(M)$. We denote by $\mathcal{F}(M)$ the norm-closed linear span of $\delta(M)$ in the dual space $\Lip(M)^*$, which is usually called the \textit{Lipschitz-free space over $M$}; for background on this, see the survey \cite{Godefroy_2015} and the book \cite{Weaver2} (where it receives the name of ``Arens-Eells space''). It is well known that $\mathcal{F}(M)$ is an isometric predual of the space $\Lip(M)$ \cite[p. 91]{Godefroy_2015}. We will write $\delta_x:=\delta(x)$ for $x\in M$, and use the name \textit{molecule} for those elements of $\mathcal F(M)$ of the form
$$m_{x,y}:=\frac{\delta_x-\delta_y}{d(x,y)}$$
for $x,y\in M$ such that $x\neq y$.

A fundamental result in the theory of Lipschitz-free spaces is that, roughly speaking, Lipschitz-free spaces linearise Lipschitz maps. In a more precise language, given a metric space $M$, a Banach space $X$ and a Lipschitz map $f:M\longrightarrow X$ such that $f(0)=0$, there exists a bounded operator $\hat f:\mathcal F(M)\longrightarrow X$ such that $\Vert \hat f\Vert=\Vert f\Vert_L$ defined by
$$\hat f(\delta_m):=f(m) \ , \quad m\in M.$$
Moreover, the mapping $f\longmapsto \hat f$ is an onto linear isometry between $\Lip(M,X)$ and the space of bounded operators $\mathcal L(\mathcal F(M),X)$.
This linearisation property makes Lipschitz-free spaces a precious magnifying glass to study Lipschitz maps between metric spaces, and for example it relates some well-known open problems in the Banach space theory to some open problems about Lipschitz-free spaces (see \cite{Godefroy_2015}). Because of this reason, the isomorphic structure of those spaces has been intensively studied in the last 20 years (see e.g. \cite{HaLaPe_2016,Kalton_2004,LanPer_2013}). In addition, the isometric structure of Lipschitz-free spaces has also been the subject of recent research (see \cite{AG_19, AP_20, GaPrRu_2018, PrRu_2018}). Results about the geometry of Lipschitz-free spaces (to be more precise, about its extremal structure) have been applied to the study of norm-attainment of Lipschitz functions \cite{CCGMR_19, Godefroy_2015, KMS_2015} and composition operators between spaces of Lipschitz functions \cite{JV_2020, Rueda_2020, Weaver2}.

In this paper we focus on the analysis of points of G\^ateaux and Fr\'echet differentiability in the unit ball of Lipschitz-free spaces.
The first results in this line appeared on \cite{BLR_2018}, where an example of a metric space $M$ is exhibited such that $\mathcal F(M)$ has a point of Fr\'echet differentiability and it is shown that this is only possible when $M$ is bounded and uniformly discrete.
This was extended in \cite[Theorem 4.3]{PrRu_2018} where it is proved that, for such $M$, a convex combination of the form $\sum_{i=1}^n \lambda_i m_{x_i,0}$ is a point of Fr\'echet differentiability if, and only if, it is a point of G\^ateaux differentiability if, and only if, $M$ is the union of the segments $[0,x_i]$.

Our goal in this paper is to extend this result to an arbitrary element $\mu\in S_{\mathcal F(M)}$ of finite support. One of the main difficulties lies in determining when a convex combination of molecules $\sum_{i=1}^n \lambda_i m_{x_i,y_i}$ has norm 1, which clearly implies that $\Vert m_{x_1,y_1}+\ldots+m_{x_n,y_n}\Vert=n$. To do so, we draw inspiration from \cite[Theorem 3.1]{PrRu_2018}, where a metric characterisation of octahedrality in $\mathcal F(M)$, which in particular involves molecules at distance almost $2$, is given in terms of a geometric condition on $M$. Motivated by that result, we prove in Theorem \ref{th:sum_of_molecules_na} a characterisation of those elements $\mu$ which are the limit of a convex series of molecules, i.e. $\mu=\sum_{n} \lambda_n m_{x_n,y_n}$ (finite or infinite sum) for which $\Vert \mu\Vert=\sum_{n} \vert \lambda_n\vert$. As an easy corollary of Theorem \ref{th:sum_of_molecules_na}, we rediscover the characterisation of sequences of molecules which are isometrically equivalent to the $\ell_1$ basis given in \cite{OsOs_19}.

In Section \ref{section:puntodif} we prove our desired result. Given a uniformly discrete and bounded metric space $M$, an element of finite support $\mu=\sum_{i=1}^n \lambda_i m_{x_i,y_i}\in S_{\mathcal F(M)}$ is a point of Fr\'echet differentiability if and only if it is a point of G\^ateaux differentiability, and this is characterised by a certain geometric condition on the pairs of points $(x_i,y_i)$ that implies, in particular, that $M$ is contained in the union of the segments spanned by these points (see Theorem \ref{th:carafrechet} for the formal statement). This extends \cite[Theorem 4.3]{PrRu_2018} to arbitrary finitely supported elements. Furthermore, we explore whether this result can be extended from elements of finite support to elements which are limit of a convex series of molecules. In this sense,
we prove a similar condition for an element $\mu=\sum_{n=1}^\infty \lambda_n m_{x_n,y_n}$ to be a point of G\^ateaux differentiability that implies that $M$ must be almost contained in the union of finitely many segments spanned by the points $x_i$ and $y_i$ (cf. Theorem \ref{theo:caragateauxarbi}).
We also provide examples, in the infinite setting, of points of Fr\'echet differentiability and points where the norm is G\^ateaux but not Fr\'echet differentiable.

Finally, in Section \ref{section:open}, we show how to apply our techniques to obtain canonical examples of points of G\^ateaux (resp. Fr\'echet) differentiability in projective tensor products $\lipfree{M}\pten X$, when $X$ is a Banach space with points of G\^ateaux (resp. Fr\'echet) differentiability and $\lipfree{M}$ has finitely supported points of Fr\'echet differentiability (see Theorem \ref{th:diferenvectorval}). Let us recall that this is not true in general for projective tensor products (cf. Remark \ref{remark:contraejetensodiffe}).

\textbf{Notation:} 
Given a metric space $M$ and two points $x,y\in M$, the sets of the form
$$[x,y]=\{z\in M: d(x,z)+d(y,z)=d(x,y)\}$$
will be called \emph{(metric) segments}. Also, given $\varepsilon>0$, we will consider the sets
$$[x,y]_\varepsilon:=\{z\in M: d(x,z)+d(y,z)<d(x,y)+\varepsilon\}.$$
We say that $M$ is \textit{uniformly discrete} if $\inf\{d(x,y): x,y\in M, x\neq y\}>0$. If $M$ is a bounded metric space, we will denote its diameter by $\diam(M)$.

We will consider only real Banach spaces, and denote by $B_X$ and $S_X$ the closed unit ball and the unit sphere of a Banach space $(X,\norm{\cdot})$. Also, we will denote by $X^*$ the topological dual of $X$. We say that $x\in X$ is a \textit{point of G\^ateaux (resp. Fr\'echet) differentiability of $X$} if the norm $\norm{\cdot}$ is G\^ateaux (resp. Fr\'echet) differentiable at $x$. By the convexity of $\norm{\cdot}$, G\^ateaux differentiability at $x$ is equivalent to the existence of the limit 
$$\lim\limits_{t\rightarrow 0}\frac{\Vert x+th\Vert-\Vert x\Vert}{t}$$
for every $h\in X$, and Fr\'echet differentiability corresponds to the limit being uniform for $h\in S_X$.
By \v Smulyan's lemma \cite[Theorem I.1.4]{dgz}, $x\in S_X$ is a point of G\^ateaux differentiability of $X$ if, and only if, there exists a unique $f\in S_{X^*}$ such that $f(x)=1$, and it is a point of Fr\'echet differentiability if, and only if
$$\inf\limits_{\alpha>0}\diam(\{f\in S_{X^*}: f(x)>1-\alpha\})=0 .$$

\section{A metric characterisation of convex series of molecules}\label{section:alcanorma}

Let $M$ be a metric space. In this section we will study the elements $\mu\in S_{\mathcal F(M)}$ which are limit in norm of a convex series; in other words, elements for which there exist a pair of sequences $(x_n), (y_n)$ in $M$ such that $x_n\neq y_n$ for every $n\in\mathbb N$ and a sequence $(\lambda_n)$ of real numbers such that 
\begin{equation}
\label{eq:convex_series}
\mu=\sum_{n=1}^\infty \lambda_n m_{x_n,y_n} \text{ where } \lambda_n\geq 0 \text{ and } \sum_{n=1}^\infty \lambda_n=1 .
\end{equation}
It is well known that any $\mu\in\sphere{\lipfree{M}}$ may be expressed as a series of type \eqref{eq:convex_series} where $\sum_{n=1}^\infty \lambda_n\leq 1+\varepsilon$, for any $\varepsilon>0$ (see e.g. \cite[Lemma 2.1]{AP_20}), but $\varepsilon=0$ is not always attainable (one such case will be described in Example \ref{example:bichogateaux}). Let us also recall that any element of $\sphere{\lipfree{M}}$ with finite support can be expressed as a finite sum of type \eqref{eq:convex_series} with $\sum_{n=1}^N \lambda_n=1$, e.g. by \cite[Proposition 3.16]{Weaver2}.

\begin{remark}\label{remark:motivacionsuma}
Elements of the form \eqref{eq:convex_series} appear naturally in certain problems about the geometry of $\mathcal F(M)$:
\begin{enumerate}[label={(\alph*)}]
    \item An open problem that goes back to the first edition of the book \cite{Weaver2} in 1999 asks whether every extreme point of $B_{\mathcal F(M)}$ is a molecule. The answer is positive for the limit of a convex series of molecules \cite[Remark 3.4]{APPP_2020}.
    \item Given a metric space $M$ and a Banach space $X$, it is said that a Lipschitz function $f:M\longrightarrow X$ \textit{strongly attains its norm} if there exists a pair of different points $x\neq y\in M$ such that
    $$\norm{f(\mol{x,y})}= \frac{\Vert f(x)-f(y)\Vert}{d(x,y)}=\Vert f\Vert_L.$$
    Denote by $\SA(M,X)$ the set of all the strongly norm attaining Lipschitz mappings and $\NA(\mathcal F(M),X)$ the set of those bounded operators from $\mathcal F(M)$ to $X$ which attain their norm. It is known that $\SA(M,X)\subseteq \NA(\mathcal F(M),X)$ and the inclusion may be strict in general (see \cite{KMS_2015}). However, if an element $f\in \NA(\mathcal F(M), X)$ attains its norm at an element which is the limit of a convex series of molecules, then an easy convexity argument shows that $f\in \SA(M,X)$.
\end{enumerate}
\end{remark}

The following general lemma will be central to our characterisation of sums of convex series in $\lipfree{M}$. It will also be used in the next section to characterise points of G\^ateaux differentiability:

\begin{lemma}
\label{lm:alpha_beta}
Let $\beta_{jk}$, $j,k\in\NN$ be real numbers such that $\beta_{jj}=0$ for every $j\in\NN$. Then there exist real numbers $\alpha_j$, $j\in\NN$ such that
\begin{equation}
    \label{eq:alpha_inequality}
    \alpha_k \leq \alpha_j+\beta_{kj}
\end{equation}
for $j,k\in\NN$ if and only if for every finite sequence $i_1,\ldots,i_m$ of natural numbers we have
\begin{equation}
    \label{eq:beta_inequality}
    \beta_{i_1i_2}+\beta_{i_2i_3}+\ldots+\beta_{i_{m-1}i_m}+\beta_{i_mi_1} \geq 0 .
\end{equation}
Moreover, the $\alpha_j$ are unique (up to an additive constant) if and only if for every pair of different numbers $j,k\in\NN$ and every $\varepsilon>0$ there is a finite sequence $i_1,\ldots,i_m$ of natural numbers that contains both $j$ and $k$ and such that
\begin{equation}
    \label{eq:beta_bound}
    \beta_{i_1i_2}+\beta_{i_2i_3}+\ldots+\beta_{i_{m-1}i_m}+\beta_{i_mi_1} \leq \varepsilon .
\end{equation}
\end{lemma}

\begin{proof}
To see that \eqref{eq:alpha_inequality} implies \eqref{eq:beta_inequality}, just notice that
$$
\sum_{k=1}^{m-1} \beta_{i_ki_{k+1}} \geq \sum_{k=1}^{m-1} (\alpha_{i_k}-\alpha_{i_{k+1}}) = \alpha_{i_1}-\alpha_{i_m} \geq -\beta_{i_mi_1} .
$$

Now suppose \eqref{eq:beta_inequality} holds, and define for $j,k\in\NN$
\begin{equation}
    \label{eq:B_jk}
    B_{jk} := \inf\set{\beta_{i_1i_2}+\beta_{i_2i_3}+\ldots+\beta_{i_{m-1}i_m}:i_1=j,i_m=k}
\end{equation}
where the infimum runs over all finite sequences $i_1,\ldots,i_m$ of natural numbers beginning with $j$ and ending with $k$. If a sequence contains a repeated index $i_r=i_s$, $r<s$ then by \eqref{eq:beta_inequality}
$$
\beta_{i_1i_2}+\ldots+\beta_{i_{m-1}i_m} \geq \beta_{i_1i_2}+\ldots+\beta_{i_{r-1}i_r}+\beta_{i_si_{s+1}}+\ldots+\beta_{i_{m-1}i_m}
$$
(we have removed the $r$-th to $(s-1)$-th terms on the right hand side), so the infimum may be restricted to sequences with no repeated indexes. Also, condition \eqref{eq:beta_inequality} implies that the infimum exists and $B_{jk}\geq -\beta_{kj}$; since $B_{jk}\leq\beta_{jk}$ by definition, we get in particular $B_{jj}=0$ for any $j$. Moreover we have
\begin{equation}
    \label{eq:B_inequality}
    B_{jl} + B_{lk} \geq B_{jk}
\end{equation}
for any $j,k,l\in\NN$: indeed, let $\varepsilon>0$ and choose finite sums such that
\begin{align*}
\beta_{ji_2}+\ldots+\beta_{i_{m-1}l} &< B_{jl}+\varepsilon \\
\beta_{li'_2}+\ldots+\beta_{i'_{m'-1}k} &< B_{lk}+\varepsilon
\end{align*}
then we have
$$
B_{jk} \leq \beta_{ji_2}+\ldots+\beta_{i_{m-1}l} + \beta_{li'_2}+\ldots+\beta_{i'_{m'-1}k} < B_{jl}+B_{lk}+2\varepsilon
$$
as this is one of the sums in the definition \eqref{eq:B_jk}.

Let us now fix any $q\in\NN$ and choose $\alpha_j=B_{jq}$ for $j\in\NN$. Then for any $j,k\in\NN$ we have
$$
\beta_{kj}+\alpha_j = \beta_{kj}+B_{jq} \geq B_{kj}+B_{jq} \geq B_{kq} = \alpha_k
$$
where we have used \eqref{eq:B_inequality}, and so \eqref{eq:alpha_inequality} holds. This proves existence. Alternatively, if we take $\alpha_j=-B_{qj}$ then
$$
\alpha_j = -B_{qj} \geq -B_{qk}-B_{kj} \geq -B_{qk}-\beta_{kj} = \alpha_k-\beta_{kj}
$$
so this choice also satisfies \eqref{eq:alpha_inequality}. Therefore, if the solution is unique up to adding a constant, then we must have $B_{jq}-B_{kq}=-B_{qj}+B_{qk}$ for any $j,k,q$, and in particular $B_{jk}+B_{kj}=0$ taking $q=k$. Conversely, let $(\alpha_j)$ be a solution of \eqref{eq:alpha_inequality}, then for any $j\neq k$ and any finite sequence $i_1,\ldots,i_m$ such that $i_1=j$, $i_m=k$ we have
$$
\alpha_j-\alpha_k = \sum_{r=1}^{m-1} (\alpha_{i_r}-\alpha_{i_{r+1}}) \leq \beta_{i_1i_2}+\ldots+\beta_{i_{m-1}i_m}
$$
and taking the infimum yields $\alpha_j-\alpha_k\leq B_{jk}$. Interchanging the role of $j$ and $k$ we get $\alpha_j-\alpha_k\geq -B_{kj}$. Therefore, if $B_{jk}+B_{kj}=0$ then the value of $\alpha_j-\alpha_k$ is really uniquely determined and equal to $B_{jk}$.

We have thus shown that the solution is unique if and only if $B_{jk}+B_{kj}=0$ for any $j,k\in\NN$. Since we always have $B_{jk}+B_{kj}\geq 0$ by \eqref{eq:beta_inequality}, this condition is equivalent to $B_{jk}+B_{kj}\leq 0$ which is clearly equivalent to \eqref{eq:beta_bound}. This ends the proof.
\end{proof}

If condition \eqref{eq:beta_bound} does not hold then the solution $(\alpha_n)$ is not unique. In fact, it is easy to check that any choice of values such that $\alpha_j-\alpha_k\in [-B_{kj},B_{jk}]$ is a valid solution, although we will not need this fact.

Lemma \ref{lm:alpha_beta} is also valid in a finite setting and the existence condition remains unchanged in that case, but we can be a bit more explicit with the uniqueness condition:

\begin{lemma}
\label{lm:alpha_beta_finite}
Let $n\in\NN$ and let $\beta_{jk}$, $j,k\in\set{1,\ldots,n}$ be real numbers such that $\beta_{jj}=0$ for every $j\in\set{1,\ldots,n}$. Then there exist real numbers $\alpha_1,\ldots,\alpha_n$, such that
$$
\alpha_k \leq \alpha_j+\beta_{kj}
$$
if and only if for every finite sequence $i_1,\ldots,i_m\in\set{1,\ldots,n}$ we have
$$
\beta_{i_1i_2}+\beta_{i_2i_3}+\ldots+\beta_{i_{m-1}i_m}+\beta_{i_mi_1} \geq 0 .
$$
Moreover, the $\alpha_j$ are unique (up to an additive constant) if and only if for every pair of different numbers $j,k\in\set{1,\ldots,n}$ there is a finite sequence of different numbers $i_1,\ldots,i_m\in\set{1,\ldots,n}$ that contains both $j$ and $k$ and such that
$$
\beta_{i_1i_2}+\beta_{i_2i_3}+\ldots+\beta_{i_{m-1}i_m}+\beta_{i_mi_1} = 0 .
$$
In that case we have $\alpha_{i_r}=\alpha_{i_{r+1}}+\beta_{i_ri_{r+1}}$ for $r=1,\ldots,m-1$ and $\alpha_{i_m}=\alpha_{i_1}+\beta_{i_mi_1}$.
\end{lemma}

\begin{proof}
The same argument is valid as in the infinite case, but there are now only finitely many sequences of different elements of $\set{1,\ldots,n}$, so the infimum in \eqref{eq:B_jk} is attained and one may take $\varepsilon=0$ in \eqref{eq:beta_bound}. Finally, notice that if $B_{jk}=\beta_{ji_2}+\ldots+\beta_{i_{m-1}k}$ is any expression minimizing \eqref{eq:B_jk} then
\begin{align*}
0 = B_{jk}+B_{kj} &= \beta_{ji_2}+\ldots+\beta_{i_{m-1}k}+B_{kj} \\
&\geq B_{ji_2}+\ldots+B_{i_{m-1}k}+B_{kj} \geq B_{jj} = 0
\end{align*}
implies that $\beta_{i_ri_{r+1}}=B_{i_ri_{r+1}}$, and this proves the last statement.
\end{proof}

We can now state the promised geometric characterisation of sums of convex series of molecules in $\lipfree{M}$:

\begin{theorem}
\label{th:sum_of_molecules_na}
Let $M$ be a pointed metric space and $(x_n,y_n)_{n\in I}$ be a finite or infinite sequence of pairs of distinct points in $M$. Then the following are equivalent:
\begin{enumerate}[label={\upshape{(\roman*)}}]
    \item\label{convexsum1} there is $f\in\sphere{\Lip(M)}$ such that $f(\mol{x_n,y_n})=1$ for every $n\in I$,
    \item\label{convexsum2} $\norm{\sum_n \lambda_n \mol{x_n,y_n}} = 1$ for some choice of $\lambda_n>0$ such that $\sum_n \lambda_n = 1$,
    \item\label{convexsum3} $\norm{\sum_n \lambda_n \mol{x_n,y_n}} = 1$ for any choice of $\lambda_n\geq 0$ such that $\sum_n \lambda_n = 1$,
    \item\label{convexsum4} for every finite sequence $i_1,\ldots,i_m$ of indices in $I$ we have
    \begin{multline}
        \label{eq:super_ltp_condition}
        d(x_{i_1},y_{i_1})+d(x_{i_2},y_{i_2})+\ldots+d(x_{i_m},y_{i_m}) \\
        \leq d(x_{i_1},y_{i_2})+d(x_{i_2},y_{i_3})+\ldots+d(x_{i_m},y_{i_1}) .
    \end{multline}
\end{enumerate}
\end{theorem}

Notice that the equivalence of properties \ref{convexsum1}-\ref{convexsum3} (which is obvious) already implies that they must be equivalent to some geometric condition on the pairs $(x_n,y_n)$ that is independent of their amount. The resulting property, described in \ref{convexsum4}, is known in optimal transport theory as \emph{cyclical monotonicity}, see e.g. \cite[Definition 5.1]{Villani}; it may also be regarded as a generalized form of the long trapezoid property (LTP) introduced in \cite{PrRu_2018}.

\begin{proof}[Proof of Theorem \ref{th:sum_of_molecules_na}]
It is clear that \ref{convexsum1}-\ref{convexsum3} are equivalent.

\ref{convexsum1}$\Rightarrow$\ref{convexsum4}: Suppose \ref{convexsum1} holds. Let $\alpha_i=f(y_i)$ and $\beta_{ij}=d(x_i,y_j)-d(x_i,y_i)$ for $i,j\in I$. By \ref{convexsum1} we have $f(x_i)=\alpha_i+d(x_i,y_i)$, and
$$
1=\lipnorm{f}\geq\frac{f(x_i)-f(y_j)}{d(x_i,y_j)}=\frac{\alpha_i-\alpha_j+d(x_i,y_i)}{d(x_i,y_j)}
$$
hence $\alpha_i\leq\alpha_j+\beta_{ij}$. Apply Lemma \ref{lm:alpha_beta} or Lemma \ref{lm:alpha_beta_finite} (depending on whether $I$ is infinite or finite) to deduce inequality \eqref{eq:beta_inequality}, which is the same as \eqref{eq:super_ltp_condition} after rearranging terms.

\ref{convexsum4}$\Rightarrow$\ref{convexsum1}: Again, let $\beta_{ij}=d(x_i,y_j)-d(x_i,y_i)$ so that inequalities \eqref{eq:super_ltp_condition} and \eqref{eq:beta_inequality} are equivalent, and use Lemma \ref{lm:alpha_beta} or Lemma \ref{lm:alpha_beta_finite} to obtain real numbers $\alpha_n$, $n\in I$ such that $\alpha_i\leq\alpha_j+\beta_{ij}$ for $i,j\in I$. Now define a function $f$ on the set $\set{x_i,y_i:i\in I}$ by $f(y_i)=\alpha_i$ and $f(x_i)=\alpha_i+d(x_i,y_i)$.

Let us first check that $f$ is well defined, i.e. that there are no conflicting assignments of values of $f$. We need to distinguish three cases:

    \vspace{0.3cm}
    
    \emph{Case 1:} Assume that $y_j=y_k$ for some $j\neq k\in I$. Then
    $$\beta_{jk}=d(x_j,y_k)-d(x_j,y_j)=d(x_j,y_j)-d(x_j,y_j)=0,$$
    hence $\alpha_j\leq \alpha_k +\beta_{jk}=\alpha_k$, and similarly $\alpha_k\leq\alpha_j$. Thus $\alpha_j=\alpha_k$ and there is no conflict.

    \vspace{0.3cm}
    
    \emph{Case 2:} Assume that $x_j=x_k$ for some $j\neq k\in I$. In this case
    \[\begin{split}
    \beta_{kj}=d(x_k,y_j)-d(x_k,y_k)& =-(d(x_k,y_k)-d(x_k,y_j))\\
    & =-(d(x_j,y_k)-d(x_j,y_j))\\
    & =-\beta_{jk}.
    \end{split}
    \]
    Now $\alpha_j\leq \alpha_k +\beta_{jk}\leq \alpha_j+\beta_{kj}+\beta_{jk}=\alpha_j,$
    so
    $$\alpha_j=\alpha_k+\beta_{jk}=\alpha_k+d(x_j,y_k)-d(x_j,y_j).$$
    Therefore $\alpha_j+d(x_j,y_j)=\alpha_k+d(x_k,y_k)$ and the definitions of $f(x_j)$ and $f(x_k)$ agree.
 
    \vspace{0.3cm}
    
    \emph{Case 3:} Assume that $y_j=x_k$ for $j\neq k\in I$. On one hand, we have
    $$\alpha_j\leq \alpha_k+\beta_{jk}=\alpha_k+d(x_j,y_k)-d(x_j,y_j)\leq \alpha_k+d(y_k,y_j)=\alpha_k+d(x_k,y_k)$$
    On the other hand
    $$\alpha_j\geq\alpha_k-\beta_{kj}=\alpha_k-(d(x_k,y_j)-d(x_k,y_k))=\alpha_k+d(x_k,y_k).$$
    Therefore $f(y_j)=\alpha_j$ and $f(x_k)=\alpha_k+d(x_k,y_k)$ do not conflict with each other.

    \vspace{0.3cm}
    
Next, let us check that $\lipnorm{f}=1$. Indeed, for $i,j\in\NN$ we have $f(\mol{x_i,y_i})=1$ and
\begin{itemize}
\item $f(x_i)-f(y_j) = \alpha_i-\alpha_j+d(x_i,y_i) \leq \beta_{ij}+d(x_i,y_i) = d(x_i,y_j)$,
\item $f(y_i)-f(x_j) = \alpha_i-\alpha_j-d(x_j,y_j) \leq \beta_{ij}-d(x_j,y_j) \leq d(y_i,x_j)$,
\item $f(x_i)-f(x_j) = \alpha_i-\alpha_j+d(x_i,y_i)-d(x_j,y_j) \leq d(x_i,x_j)$,
\item $f(y_i)-f(y_j) = \alpha_i-\alpha_j \leq \beta_{ij} \leq d(y_i,y_j)$
\end{itemize}
as is straightforward to check. Now use McShane's theorem (see e.g. \cite[Theorem 1.33]{Weaver2}) to extend $f$ to a 1-Lipschitz function on $M$ and subtract a constant so that $f(0)=0$. The resulting $f\in\sphere{\Lip(M)}$ still satisfies $f(\mol{x_i,y_i})=1$ for every $i\in I$ as required.
\end{proof}

\begin{remark}
Since any finitely supported element $\mu\in\sphere{\lipfree{M}}$ can be written as a finite sum $\sum_n \lambda_n \mol{x_n,y_n}$ where $\sum_n \lambda_n = 1$ and $x_n,y_n$ sweep over the support of $\mu$, Theorem \ref{th:sum_of_molecules_na} implies that it is always possible to organize any given finite subset of points of $M$ into pairs $(x_n,y_n)$ in such a way that \eqref{eq:super_ltp_condition} holds.
\end{remark}

As a first application of Theorem \ref{th:sum_of_molecules_na}, we can obtain a precise description of those sequences of molecules that are isometrically equivalent to the $\ell_1$ basis. To achieve that, we use the following standard lemma that we state without proof:

\begin{lemma}\label{lema:basesl1}
Let $X$ be a Banach space and let $(x_n)$ be a sequence in $S_X$. The following assertions are equivalent:
\begin{enumerate}[label={\upshape{(\roman*)}}]
    \item\label{lema:basesl12} The sequence $(x_n)$ is isometrically equivalent to the $\ell_1$ basis.
    \item\label{lema:basesl11} For every sequence $(\sigma(n))$ in $\{-1,1\}$ we can find a functional $f\in S_{X^*}$ such that $f(x_n)=\sigma(n)$ for every $n\in\mathbb N$.
\end{enumerate}
\end{lemma}

Combining this with Theorem \ref{th:sum_of_molecules_na}, the following is immediate: 

\begin{corollary}\label{theo:carabasel1}
Let $M$ be a pointed metric space and $(x_n,y_n)$ be a sequence of pairs of different points in $M$. The following are equivalent:
\begin{enumerate}[label={\upshape{(\roman*)}}]
    \item The sequence $(m_{x_n,y_n})$ is isometrically equivalent to the $\ell_1$ basis.
    \item For every sequence $(\sigma(n))$ in $\{-1,1\}$ there exists $f\in\sphere{\Lip(M)}$ such that $f(m_{x_n,y_n})=\sigma(n)$ for every $n\in\NN$.
    \item For every choice of $\{u_k,v_k: k\in\mathbb N\}$ such that $\{u_k,v_k\}=\{x_k,y_k\}$ for every $k\in\mathbb N$, and for every finite sequence of numbers $i_1,\ldots,i_m$ in $\NN$, we have
    \begin{multline*}
	d(u_{i_1},v_{i_1})+d(u_{i_2},v_{i_2})+\ldots+d(u_{i_m},v_{i_m}) \\
	\leq d(u_{i_1},v_{i_2})+d(u_{i_2},v_{i_3})+\ldots+d(u_{i_m},v_{i_1}) .
    \end{multline*}
\end{enumerate}
\end{corollary}

The equivalence (i)$\Leftrightarrow$(iii) is exactly the same one that is given (using a different terminology) in Theorem 2.1 of the recent preprint \cite{OsOs_19}, which is in turn based on a result from \cite{KMO_19}. Note that we only characterise $\ell_1$ bases of molecules while the result in \cite{OsOs_19} is a bit more general, as it also says that whenever $\lipfree{M}$ contains an isometric $\ell_1$ basis, it must contain in particular one consisting only of molecules.

\section{Points of G\^ateaux and Fr\'echet differentiability}\label{section:puntodif}

In this section we will study necessary and sufficient conditions for the limit of a convex series of molecules of the form \eqref{eq:convex_series} to be a point of G\^ateaux (resp. Fr\'echet) differentiability of the norm of $\lipfree{M}$. Consequently, throughout the section, when we write $\mu=\sum_{n} \lambda_n m_{x_n,y_n}$ we will assume that $\mu$ is the limit of a convex series of molecules, i.e. the previous series is norm convergent and $\Vert \mu\Vert=\sum_n \lambda_n$. The sum may be finite or infinite, but in any case we will assume without loss of generality that $\lambda_n>0$ for all $n$. So, according to Theorem \ref{th:sum_of_molecules_na}, the sequences of points $(x_n)$ and $(y_n)$ will satisfy \eqref{eq:super_ltp_condition}. We will make use of the previous fact without any further reference.

With this notation in mind, let us begin by looking for a characterisation of the fact that $\mu$ is a point of G\^ateaux differentiability in terms of a geometric condition on the sequences of points $(x_n), (y_n)$.
Note first that this happens if and only if there is a unique $f\in\sphere{\Lip(M)}$ such that $f(\mu)=1$.
In view of the proof of Theorem \ref{th:sum_of_molecules_na}, we will need to involve the condition of uniqueness appearing in Lemma \ref{lm:alpha_beta}. If that condition holds, then $f$ will be uniquely defined in $\bigcup_{n} [x_n,y_n]$: indeed, if $f(m_{x_i,y_i})=1$ then $f(m_{x_i,z})=f(m_{z,y_i})=1$ for every element $z\in [x_i,y_i]$ (see e.g. \cite[Lemma 2.2]{KMS_2015}). Thus, if we require to $M$ be contained in $\bigcup_{n} [x_n,y_n]$, then the uniqueness of the Lipschitz function strongly attaining its norm simultaneously at every $m_{x_n,y_n}$ should imply that $\mu$ is a point of G\^ateaux differentiability. 

The previous remarks suggest the main idea behind the following result.

\begin{theorem}\label{theo:caragateauxarbi}
Let $M$ be a pointed metric space,
let $\mu\in\sphere{\lipfree{M}}$ be of the form \eqref{eq:convex_series},
and let $f\in\sphere{\Lip(M)}$ be such that $f(\mu)=1$. Then $\mu$ is a point of G\^ateaux differentiability if and only if the following two conditions hold for every $\varepsilon>0$:
\begin{enumerate}[label={\upshape{(\roman*)}}]
\item for every pair of different numbers $j,k\in\NN$ there is a finite sequence $i_1,\ldots,i_m$ in $\NN$ that contains $j$ and $k$ and such that
\begin{multline*}
d(x_{i_1},y_{i_1})+d(x_{i_2},y_{i_2})+\ldots+d(x_{i_m},y_{i_m}) \\
> d(x_{i_1},y_{i_2})+d(x_{i_2},y_{i_3})+\ldots+d(x_{i_m},y_{i_1}) - \varepsilon ,
\end{multline*}
\item for every $x\in M$ there are $s,t\in\set{x_1,y_1,x_2,y_2,\ldots}$ such that $x\in [s,t]_\varepsilon$ and $f(t)-f(s)>d(t,s)-\varepsilon$.
\end{enumerate}
\end{theorem}

\begin{proof}
Denote $N=\set{x_1,y_1,x_2,y_2,\ldots}$, $\alpha_i=f(y_i)$ and $\beta_{ij}=d(x_i,y_j)-d(x_i,y_i)$.
Recall that $\mu$ is a point of G\^ateaux differentiability if and only if $f$ is unique. So assume that $f$ is unique, then condition (i) follows immediately from Lemma \ref{lm:alpha_beta}. To see condition (ii), let $g_1$ and $g_2$ be the largest and smallest 1-Lipschitz extensions of $f\restrict_N$ to $M$, respectively, given by
\begin{equation}
\label{eq:extensions_from_N}
\begin{split}
g_1(x) &= \inf_{p\in N}\pare{f(p)+d(p,x)} \\
g_2(x) &= \sup_{p\in N}\pare{f(p)-d(p,x)}
\end{split}
\end{equation}
for $x\in M$, and note that $g_1=g_2$ by assumption. Now fix $x\in M$ and let $s,t\in N$ be such that $g_1(x)+\frac{\varepsilon}{2}>f(t)+d(t,x)$ and $g_2(x)-\frac{\varepsilon}{2}<f(s)-d(s,x)$. Then
\begin{align*}
\varepsilon &> f(t)+d(t,x)-g_1(x)+g_2(x)-f(s)+d(s,x) \\
&= f(t)-f(s)+d(t,x)+d(s,x) \\
&\geq d(t,x)+d(s,x)-d(t,s)
\end{align*}
therefore $x\in [s,t]_\varepsilon$, and $f(s)-f(t)>d(t,x)+d(s,x)-\varepsilon\geq d(s,t)-\varepsilon$.

Now assume that conditions (i) and (ii) hold. Let $x\in M$, $\varepsilon>0$, then by (ii) there are $s,t\in N$ such that $x\in [s,t]_\varepsilon$ and $f(t)-f(s)>d(t,s)-\varepsilon$. Let $g\in\sphere{\Lip(M)}$ be such that $g(\mu)=1$.  By (i) and Lemma \ref{lm:alpha_beta} we have $g\restrict_N=f\restrict_N$ and in particular $g(s)=f(s)$ and $g(t)=f(t)$. Therefore
\begin{align*}
1 &\geq \frac{d(x,s)}{d(x,s)+d(x,t)}g(\mol{s,x}) + \frac{d(x,t)}{d(x,s)+d(x,t)} g(\mol{x,t}) \\
&= \frac{d(s,t)}{d(x,s)+d(x,t)}g(\mol{s,t}) \\
&= \frac{d(s,t)}{d(x,s)+d(x,t)}f(\mol{s,t}) > \frac{d(t,s)-\varepsilon}{d(x,s)+d(x,t)} > 1-\frac{2\varepsilon}{d(s,t)+\varepsilon}
\end{align*}
and it follows by a standard convexity argument that $g(\mol{s,x})$, $g(\mol{x,t})$ are both larger than $1-\varepsilon'$, where
$$
\varepsilon' = \frac{d(x,s)+d(x,t)}{\min\set{d(x,s),d(x,t)}}\cdot\frac{2\varepsilon}{d(s,t)+\varepsilon} < \frac{2\varepsilon}{\min\set{d(x,s),d(x,t)}} .
$$
From $g(\mol{s,x})>1-\varepsilon'$ and $g(s)=f(s)$ we get
$$
f(s)-d(s,x)\leq g(x) < f(s)-(1-\varepsilon')d(s,x) = f(s)-d(s,x)+\varepsilon' d(s,x)
$$
and, since this also applies to the case $g=f$, we have $\abs{g(x)-f(x)}<\varepsilon' d(s,x)$. A similar argument shows that $g(\mol{x,t})>1-\varepsilon'$ and $g(t)=f(t)$ imply that $\abs{g(x)-f(x)}<\varepsilon' d(t,x)$. Taking the minimum of these two bounds yields
$$
\abs{g(x)-f(x)} < \varepsilon'\cdot\min\set{d(x,s),d(x,t)} < 2\varepsilon .
$$
Since $x$, $\varepsilon$ and $g$ were arbitrary, we conclude that $f$ is unique and this ends the proof.
\end{proof}

We have thus characterised those points of G\^ateaux differentiability that may be expressed in the form \eqref{eq:convex_series}. The following example shows that not all points of G\^ateaux differentiability may be written as the limit of a convex series of molecules:

\begin{example}\label{example:bichogateaux}
Let $M=[0,1]$. It is well known (see e.g. \cite{Godefroy_2015}) that there is an onto linear isometry $\Phi: L_1([0,1])\longrightarrow \mathcal F(M)$ given by $\Phi(\chi_{[0,x]})=\delta_x$ for $x\in M$ (where $\chi$ denotes a characteristic function), whose adjoint operator $\Phi^*:\Lip(M)\longrightarrow L_\infty([0,1])$ is given by $\Phi(f)=f'$. Let $C$ be a nowhere dense closed subset of $[0,1]$ with positive Lebesgue measure, e.g. a ``fat Cantor set'' (see e.g. \cite[Example 1.40]{Weaver2} for one possible construction) and denote $D=[0,1]\setminus C$. Define $f=\chi_C-\chi_D\in S_{L_1([0,1])}$ and let $\mu=\Phi(f)\in S_{\mathcal F(M)}$. We claim that $\mu$ is a point of G\^ateaux differentiability that cannot be written as the limit of a convex series of molecules. This will be proved in two steps:

\vspace{0.3cm}
\emph{Step 1:} $\mu$ is a point of G\^ateaux differentiability.
\vspace{0.3cm}

Since $\Phi$ is an onto isometry, it is enough to prove that $f$ is a point of G\^ateaux differentiability in $L_1([0,1])$.
But notice that if $g\in S_{L_\infty([0,1])}$ satisfies that $\duality{f,g}=\int_0^1 gf\,dm=1$ (where $m$ is the Lebesgue measure) then $g=1$ a.e. on $C$ and $g=-1$ a.e. on $D$. So $g=f$ is unique in $S_{L_\infty([0,1])}$ and \v Smulyan's lemma yields the desired result.

\vspace{0.3cm}
\emph{Step 2:} $\mu$ is not the limit of a convex series of molecules.
\vspace{0.3cm}

Assume for contradiction that $\mu=\sum_{n=1}^\infty \lambda_n m_{x_n,y_n}$ where $x_n<y_n$, $\sum_{n=1}^\infty \abs{\lambda_n}=1$ and $\lambda_n\neq 0$ (but they may be negative). Denote $I_n=[x_n,y_n]$. By the definition of $\Phi$ we have
\[
    f =\Phi^{-1}(\mu) = \sum_{n=1}^\infty \lambda_n \Phi^{-1}(m_{x_n,y_n})
    = \sum_{n=1}^\infty (-\lambda_n) \frac{\chi_{I_n}}{d(x_n,y_n)} \,.
\]
Evaluating against $g=f$ seen as an element of $S_{L_\infty([0,1])}$, we get that
$$
1 = \duality{f,g} = \sum_{n=1}^\infty (-\lambda_n) \frac{1}{d(x_n,y_n)}\int_{x_n}^{y_n} f\,dm = \sum_{n=1}^\infty \lambda_n \frac{m(I_n\cap D)-m(I_n\cap C)}{d(x_n,y_n)} .
$$
Taking into account that each term multiplying $\lambda_n$ has absolute value less or equal to $1$, we get that
$$
\frac{m(I_n\cap D)-m(I_n\cap C)}{d(x_n,y_n)}=\operatorname{sign}(\lambda_n)
$$
for all $n$. Notice that, since $D$ is open and dense, $m(I_n\cap D)>0$ for every $n\in\mathbb N$. Therefore
\begin{align*}
\operatorname{sign}(\lambda_n)\cdot d(x_n,y_n) &= m(I_n\cap D)-m(I_n\cap C) \\
&> -m(I_n\cap D)-m(I_n\cap C) = -m(I_n) = -d(x_n,y_n)
\end{align*}
and thus $\lambda_n>0$ for every $n\in\mathbb N$. But then
$$
m(I_n\cap D)+m(I_n\cap C) = m(I_n) = d(x_n,y_n) = m(I_n\cap D)-m(I_n\cap C)
$$
shows that $m(I_n\cap C)=0$ for every $n\in\mathbb N$. Now define $h=\chi_C\in L_\infty([0,1])$. Then $\duality{\chi_{I_n},h} = \int_{x_n}^{y_n} \chi_C\,dm = m(I_n\cap C)=0$, and so
$$
0 = \duality{\sum_{n=1}^\infty (-\lambda_n) \frac{\chi_{I_n}}{d(x_n,y_n)},h} = \duality{f,h} = m(C) ,
$$
a contradiction. Consequently, $\mu$ is not the limit of a convex series of molecules, as claimed.
\end{example}

\begin{remark}\label{remark:miguel}
It is also possible to construct an example of point of G\^ateaux differentiability in $\lipfree{[0,1]}$ by making use of the theory of strong norm attainment (see Remark \ref{remark:motivacionsuma}). To this end, pick a measurable set $A\subset [0,1]$ such that $0<m(A\cap I)<m(I)$ holds for every open interval $I\subset [0,1]$. Let $f=\chi_A-\chi_{[0,1]\setminus A}\in L_1([0,1])$ and $\mu=\Phi(f)$, using the notation of Example \ref{example:bichogateaux}. Then the same argument shows that $\mu$ is a point of G\^ateaux differentiability of $\mathcal F(M)$.
However, $\mu$ is not the sum of a convex series of molecules, because otherwise $\Phi^{-1}(f)$ would be a strongly norm-attaining Lipschitz function, contradicting the fact that $d(f,\SA([0,1],\mathbb R))=1$ as is proved in \cite[p. 109]{Godefroy_2015}. We thank Miguel Mart\'in for pointing out this example to us.
\end{remark}

In view of \cite[Theorem 4.3]{PrRu_2018}, we may ask whether, under the assumptions of Theorem \ref{theo:caragateauxarbi}, $\mu$ is actually a point of Fr\'echet differentiability.
Let us recall that it is only possible for points of Fr\'echet differentiability to exist in $\lipfree{M}$ when $M$ is uniformly discrete and bounded \cite[Theorem 2.4]{BLR_2018}.
The following example reveals that even in that case the answer to our question is negative in general. This proves that an extension of \cite[Theorem 4.3]{PrRu_2018} to the infinitely supported setting would be false.

\begin{example}\label{example:ganofre}
Consider $M:=\mathbb N\cup\{0\}$ with $d(n,0)=1$ and $d(m,n)=2$ for every $m\neq n\in\NN$. Then $\mathcal F(M)$ is isometrically isomorphic to $\ell_1$, so the unit ball of $\mathcal F(M)$ does not have any point of Fr\'echet differentiability (cf. e.g. \cite[Example I.1.6(c)]{dgz}). However $\mu=\sum_{n=1}^\infty \frac{\delta_n}{2^n}$ is a point of G\^ateaux differentiability, as it is clear that if $f(\mu)=1$ and $f\in S_{\Lip(M)}$ then $f(n)=1$ holds for every $n\in\mathbb N$.
\end{example}

In spite of this example, it is possible to extend \cite[Theorem 4.3]{PrRu_2018} and show that Fr\'echet and G\^ateaux differentiability are indeed equivalent in the finitely supported setting:

\begin{theorem}
\label{th:carafrechet}
Let $M$ be a uniformly discrete, bounded pointed metric space and let $\mu\in\sphere{\lipfree{M}}$ be finitely supported. Write $\mu$ as a finite sum of the form \eqref{eq:convex_series} and let $f\in\sphere{\Lip(M)}$ be such that $f(\mu)=1$. The following assertions are equivalent:
\begin{enumerate}[label={\upshape{(\roman*)}}]
	\item\label{th:carafrechet1} $\mu$ is a point of Fr\'echet differentiability,
    \item\label{th:carafrechet2} $\mu$ is a point of G\^ateaux differentiability,
    \item\label{th:carafrechet3} for every pair of different numbers $j,k\in\set{1,\ldots,n}$ there is a finite sequence of different numbers $i_1,\ldots,i_m\in\set{1,\ldots,n}$ that contains both $j$ and $k$ and such that
    \begin{multline}
    \label{eq:suma_betas_0}
    d(x_{i_1},y_{i_1})+d(x_{i_2},y_{i_2})+\ldots+d(x_{i_m},y_{i_m}) \\
    = d(x_{i_1},y_{i_2})+d(x_{i_2},y_{i_3})+\ldots+d(x_{i_m},y_{i_1}) ,
    \end{multline}
    and for every $x\in M$ there are $s\neq t$ in $\set{x_1,y_1,\ldots,x_n,y_n}$ such that $f(t)-f(s)=d(t,s)$ and $x\in [s,t]$.
\end{enumerate}
\end{theorem}

\begin{proof}
Denote $D=\diam(M)$, $\theta=\inf\limits_{x\neq y\in M} d(x,y)>0$, $N=\set{x_1,y_1,\ldots,x_n,y_n}$, $\alpha_i=f(y_i)$ and $\beta_{ij}=d(x_i,y_j)-d(x_i,y_i)$. It is clear that \ref{th:carafrechet1}$\Rightarrow$\ref{th:carafrechet2}. To prove that \ref{th:carafrechet2}$\Rightarrow$\ref{th:carafrechet3}, notice first that \eqref{eq:suma_betas_0} follows from Lemma \ref{lm:alpha_beta_finite}. Now fix $x\in M$, then the argument used in the proof of Theorem \ref{theo:caragateauxarbi} shows that for any $\varepsilon>0$ there are $s\neq t\in N$ such that $x\in [s,t]_\varepsilon$ and $f(t)-f(s)>d(t,s)-\varepsilon$. Since there are only finitely many possible choices for the pair $(s,t)$, one of them must be valid for arbitrarily small values of $\varepsilon$, and so \ref{th:carafrechet3} follows.

Finally, let us see that \ref{th:carafrechet3}$\Rightarrow$\ref{th:carafrechet1}. To this end, assume with no loss of generality that $y_1=0$.
Take $g\in S_{\Lip(M)}$ such that
$$g(\mu)>1-\frac{\varepsilon}{\min\limits_{1\leq i\leq n}\lambda_i},$$
which implies by a convexity argument that $g(m_{x_i,y_i})>1-\varepsilon$, and hence
\begin{equation}\label{equa:frechminisalto}
0 \leq d(x_i,y_i)-(g(x_i)-g(y_i)) < \varepsilon d(x_i,y_i) \leq D\varepsilon
\end{equation}
for every $i\in\set{1,\ldots,n}$.
Now pick $k\in\{2,\ldots, n\}$. By assumption there exists a finite sequence of different numbers $i_1,\ldots, i_m\in\{1,\ldots, n\}$ containing both $1$ and $k$ such that \eqref{eq:suma_betas_0} holds; assume $i_1=1$ without loss of generality.
Then we have
\[
1=\Vert g\Vert_L \geq \frac{g(x_{i_m})-g(y_{i_1})}{d(x_{i_m},y_{i_1})} = \frac{\displaystyle \sum_{t=1}^m \pare{g(x_{i_t})-g(y_{i_t})} + \sum_{t=1}^{m-1}\pare{g(y_{i_{t+1}})-g(x_{i_t})}}{d(x_{i_m},y_{i_1})}
\]
hence
\[
\begin{split}
d(x_{i_m},y_{i_1}) &\geq \sum_{t=1}^m \pare{g(x_{i_t})-g(y_{i_t})} + \sum_{t=1}^{m-1} \pare{g(y_{i_{t+1}})-g(x_{i_t})} \\
&> (1-\varepsilon)\sum_{t=1}^m d(x_{t_i},y_{t_i}) + \sum_{t=1}^{m-1} \pare{g(y_{i_{t+1}})-g(x_{i_t})} \\
&= -\varepsilon\sum_{t=1}^m d(x_{t_i},y_{t_i}) + \sum_{t=1}^{m-1}  \pare{d(x_{i_t},y_{i_t})-d(x_{i_t},y_{i_{t+1}})} \\
&\qquad + d(x_{i_m},y_{i_m}) + \sum_{t=1}^{m-1} \pare{d(x_{i_t},y_{i_{t+1}})-(g(x_{i_t})-g(y_{i_{t+1}}))} .
\end{split}
\]
Rearranging terms and applying \eqref{eq:suma_betas_0} we get
\[
\sum_{t=1}^{m-1} \pare{d(x_{i_t},y_{i_{t+1}})-(g(x_{i_t})-g(y_{i_{t+1}}))} < \varepsilon \sum_{t=1}^m d(x_{i_t},y_{i_t}) < mD\varepsilon .
\]
 Since $d(x_{i_t},y_{i_{t+1}})-(g(x_{i_t})-g(y_{i_{t+1}}))\geq 0$ holds for every $1\leq t\leq m-1$ we obtain
\begin{equation}\label{equa:frechidesalto}
0 \leq d(x_{i_t},y_{i_{t+1}})-(g(x_{i_t})-g(y_{i_{t+1}})) < mD\varepsilon .
\end{equation}
Now notice that this reasoning is also valid for the function $f$ in place of $g$, and therefore \eqref{equa:frechminisalto} and \eqref{equa:frechidesalto} imply that
\begin{align*}
\abs{(f-g)(x_{i_t})-(f-g)(y_{i_t})} &< D\varepsilon \\
\abs{(f-g)(x_{i_t})-(f-g)(y_{i_{t+1}})} &< mD\varepsilon
\end{align*}
for every $t\leq m-1$. But $f(y_1)=0=g(y_1)$, so after at most $m-1$ applications of these inequalities we get that

$$
\abs{f(u)-g(u)} < (m-1)(m+1)D\varepsilon < m^2 D\varepsilon \leq n^2 D\varepsilon
$$
for $u\in\set{x_k,y_k}$. Since $k$ was arbitrary, the inequality holds for every $u\in N$.

Now pick $x\in M$. By assumption there exists a pair of different points $s,t\in N$ such that $f(m_{s,t})=1$ and $x\in [s,t]$ and we get that
$$g(m_{s,t}) =1-(f-g)(m_{s,t}) >1-\frac{2n^2D}{\theta}\varepsilon.$$
Let us take into account that $x\in [s,t]$ implies that $m_{s,t}=\lambda_1 m_{s,x}+\lambda_2 m_{x,t}$ for some $\lambda_1,\lambda_2\geq 0$ with $\lambda_1+\lambda_2=1$. Assume with no loss of generality that $\lambda_1\geq \frac{1}{2}$, then we get that
$$g(m_{s,x})>1-\frac{4n^2D}{\theta}\varepsilon.$$
But $\vert (f-g)(s)\vert<n^2D \varepsilon$ and $f(m_{s,x})=1$, so a convexity argument yields $\vert (f-g)(x)\vert<\left(\frac{4}{\theta}+1\right) n^2 D \varepsilon$.
Since $x\in M$ was arbitrary we deduce that
$$\Vert f-g\Vert_\infty\leq \left(\frac{4}{\theta}+1\right) n^2 D\cdot \varepsilon.$$
Finally, since $M$ is uniformly discrete and bounded, $\Vert\cdot\Vert_L$ and $\Vert \cdot\Vert_\infty$ are equivalent norms on $\Lip(M)$, so there exists a constant $K>0$ that depends only on $M$ and $n$ such that $\Vert f-g\Vert_L<K\varepsilon$.

Summarising: we have proved that, given any $\varepsilon>0$, there exists $\delta>0$ such that if $g\in S_{\Lip(M)}$ satisfies that $g(\mu)>1-\delta$ then $\Vert f-g\Vert_L<\varepsilon$. By \v{S}mulyan's lemma $\mu$ is a point of Fr\'echet differentiability and we are done.
\end{proof}

The previous theorem shows that, when dealing with finitely supported elements in Lipschitz free spaces over uniformly discrete and bounded metric spaces, G\^ateaux and Fr\'echet differentiability are equivalent. Example \ref{example:ganofre} reveals that this is not the case when dealing with elements of infinite support. A closer look at this example shows that the metric space is still union of metric segments; the failure of the Fr\'echet differentiability comes now from the fact that there are infinitely many segments which are uniformly separated. This phenomenon will become clear with the following result.

\begin{proposition}
\label{pr:frechet_segment_nbhd}
Let $M$ be a uniformly discrete, bounded pointed metric space,
and let $\mu\in\sphere{\lipfree{M}}$ be of the form \eqref{eq:convex_series}.
Suppose that $\mu$ is a point of Fr\'echet differentiability. Then for every $\varepsilon>0$ there is $n\in\NN$ such that
\begin{equation}
\label{eq:inclusion_frechet}
M\subset\bigcup_{s,t\in\set{x_1,y_1,\ldots,x_n,y_n}} [s,t]_\varepsilon .
\end{equation}
More precisely, one can restrict the union to those pairs $s,t$ such that $f(s)-f(t)>d(s,t)-\varepsilon$, where $f\in\sphere{\Lip(M)}$ is the Fr\'echet derivative of the norm at $\mu$.
\end{proposition}

\begin{proof}
Since the Lipschitz and supremum norms are equivalent in $\Lip(M)$, we may find $\delta>0$ such that $\norm{f-g}_\infty< \frac{\varepsilon}{2}$ whenever $g\in\sphere{\Lip(M)}$ satisfies $g(\mu)>1-2\delta$. Choose $n\in\NN$ such that $\sum_{k>n}\lambda_k<\delta$ and denote $N=\set{x_1,y_1,\ldots,x_n,y_n}$. Now let $g_1$ and $g_2$ be the largest and smallest 1-Lipschitz extensions of $f\restrict_N$ to $M$, respectively, given by \eqref{eq:extensions_from_N}.
Notice that
\begin{align*}
g_i(\mu) &= \sum_{k=1}^n \lambda_k f(\mol{x_k,y_k}) + \sum_{k>n} \lambda_k g_i(\mol{x_k,y_k}) \\
&= 1 + \sum_{k>n} \lambda_k (g_i-f)(\mol{x_k,y_k}) \geq 1-2\sum_{k>n} \lambda_k > 1-2\delta
\end{align*}
for $i=1,2$, therefore $\norm{g_1-g_2}_\infty<\varepsilon$. Now fix $x\in M$ and let $s,t\in N$ be such that $g_1(x)=f(t)+d(t,x)$ and $g_2(x)=f(s)-d(s,x)$. Then
$$
\varepsilon > g_1(x)-g_2(x) = f(t)-f(s)+d(t,x)+d(s,x) \geq d(t,x)+d(s,x)-d(t,s)
$$
and therefore $x\in [s,t]_\varepsilon$, and $f(s)-f(t)>d(t,x)+d(s,x)-\varepsilon\geq d(s,t)-\varepsilon$.
\end{proof}

We may wonder whether the conclusion of Proposition \ref{pr:frechet_segment_nbhd} holds in an uniform way, i.e. whether we can find $n\in\NN$ such that the inclusion \eqref{eq:inclusion_frechet} holds for arbitrarily small $\varepsilon$. This is not possible in general: indeed, using the same argument as in the proof of \ref{th:carafrechet2}$\Rightarrow$\ref{th:carafrechet3} in Theorem \ref{th:carafrechet}, it is easy to see that this would imply that $M$ can be covered by finitely many segments $[s,t]$ such that $s\neq t\in\set{x_1,y_1,x_2,y_2,\ldots}$ and $f(\mol{s,t})=1$.

We finish the section by exhibiting a point of Fr\'echet differentiability of infinite support.
We will construct an element of the form $\mu=\sum_{n=1}^\infty \lambda_n m_{x_n,0}$. Our strategy will be to define the sequence $(x_n)$ in such a way that $x_n$ approaches the segment $[0,x_1]$ as $n$ increases.
Consequently, if $g(\mu)$ is large, then $g(x_n)$ should behave like $d(x_n,0)$ for small $n$ and it should be almost determined for large $n$ because $x_n$ is then close to $[0,x_1]$.
For such behavior, we consider in the following example the metric space defined in \cite[Example 4.3]{AG_19}:

\begin{example}\label{examp:frechetsopinf}
Let $M:=\{0\}\cup\{x_n:n\in\mathbb N\}\subseteq c_0$ where $x_1:=2e_1$ and $x_n:=e_1+(1+\frac{1}{2^n})e_n$ for $n\geq 2$, and
define $\mu:=\sum_{n=1}^\infty \frac{1}{2^n} m_{x_n,0}$. Notice that $\mu\in S_{\mathcal F(M)}$. Indeed, consider $f(x):=d(x,0)$, and note that $f(\mu)=1$. Let us prove that $\mu$ is a point of Fr\'echet differentiability. To do so pick an arbitrary $n\geq 3$, define $\varepsilon:=\frac{1}{2^n}$ and assume that $g\in S_{\Lip(M)}$ satisfies that $g(\mu)=\sum_{n=1}^\infty \frac{1}{2^n}g(m_{x_n,0})>1-\varepsilon^2$. Let
$$P:=\{k\in\mathbb N: g(m_{x_k,0})\leq 1-\varepsilon\}.$$
Then $\sum_{k\in P} \frac{1}{2^k}<\varepsilon$: indeed,
\[
\begin{split}
1-\varepsilon^2<\sum_{k=1}^\infty \frac{1}{2^k}g(m_{x_k,0})& =\sum_{k\in P} \frac{1}{2^k}g(m_{x_k,0})+\sum_{k\notin P}\frac{1}{2^k}g(m_{x_k,0})\\
& \leq (1-\varepsilon)\sum_{k\in P} \frac{1}{2^k}+\sum_{k\notin P} \frac{1}{2^k}\\
& =\sum_{k=1}^\infty \frac{1}{2^k}-\varepsilon \sum_{k\in P}\frac{1}{2^k}\\
& = 1-\varepsilon\sum_{k\in P}\frac{1}{2^k},
\end{split}
\]
as desired. Let us estimate $\Vert f-g\Vert_\infty$ from the above inequality. On the one hand, if $k\notin P$, we get that $f(m_{x_k,0})=1$ and $g(m_{x_k,0})>1-\varepsilon$, from where $\vert (f-g)(m_{x_k,0})\vert<\varepsilon$, so
$$\vert (f-g)(x_k)\vert\leq \varepsilon d(x_k,0)\leq 2\varepsilon.$$
On the other hand, if $k\in P$ notice that $k>n$ because $\sum_{k\in P}\frac{1}{2^k}<\varepsilon=\frac{1}{2^n}$.
In particular $1\notin P$ and so we have
\[
\begin{split}
    1-\varepsilon<\frac{g(x_1)-g(0)}{d(x_1,0)}& =\frac{g(x_1)-g(x_k)+g(x_k)-g(0)}{2}\\
    & \leq \frac{d(x_1,x_k)+g(x_k)}{2}\\
    & =\frac{1+\frac{1}{2^k}+g(x_k)}{2},
\end{split}
\]
and we get 
$$g(x_k)>1-2\varepsilon-\frac{1}{2^k}=1-\frac{1}{2^{n-1}}-\frac{1}{2^k}.$$
Since $f(x_k)=1+\frac{1}{2^k}$, we get that 
$$\vert (f-g)(x_k)\vert<\frac{1}{2^{n-1}}+\frac{2}{2^k}\leq \frac{1}{2^{n-2}}.$$
To sum up, we have proved that, for every $n\geq 3$, if $g\in S_{\Lip(M)}$ satisfies that $g(\mu)>1-\frac{1}{2^{2n}}$ then $\Vert f-g\Vert_\infty<\frac{1}{2^{n-2}}$. Now \v{S}mulyan's lemma, together with the fact that the $\lipnorm{\cdot}$ and $\norm{\cdot}_\infty$ norms in $\Lip(M)$ are equivalent, implies that $\mu$ is a point of Fr\'echet differentiability.
\end{example}

\section{Remarks and open questions}\label{section:open}

As we pointed out in the Introduction, we will apply the techniques of Theorem \ref{th:carafrechet} to obtain a result of differentiability in a projective tensor product. In order to do so, let us introduce a bit of notation. Given two Banach spaces $X$ and $Y$, recall that the \textit{projective tensor product} of $X$ and $Y$, denoted by $X\pten Y$, is the completion of $X\otimes Y$ under the norm given by
$$
   \Vert u \Vert :=
   \inf\left\{
      \sum_{i=1}^n  \Vert x_i\Vert\Vert y_i\Vert
      : u=\sum_{i=1}^n x_i\otimes y_i
      \right\}.
$$
It is well known that, given two Banach spaces $X$ and $Y$, then
$(X\pten Y)^*=\mathcal L(X,Y^*)$ (see \cite{rya} for background).

Let $M$ be a uniformly discrete and bounded metric space and $Z$ be a Banach space. In this context, notice that $(\mathcal F(M)\pten Z)^*=\mathcal L(\mathcal F(M),Z^*)=\Lip(M,Z^*)$. Let $\mu=\sum_{i=1}^n \lambda_i m_{x_i,y_i}\in S_{\mathcal F(M)}$ where $\lambda_i>0$ and $\sum_{i=1}^n\lambda_i=1$, and take $z\in S_Z$. Assume that both $\mu$ and $z$ are points of Fr\'echet differentiability, with respective derivatives $f\in S_{\Lip(M)}$ and $z^*\in S_{Z^*}$. We claim that $\mu\otimes z$ is then a point of Fr\'echet differentiability in $\mathcal F(M)\pten Z$ and that its derivative is $f\otimes z^*\in \Lip(M,Z^*)$. To prove it, pick $\varepsilon>0$ assume that $g\in S_{\Lip(M,Z^*)}$ satisfies that
$$g(m_{x_i,y_i})(z)\geq 1-\varepsilon$$
holds for every $i\in\{1,\ldots, n\}$. Following the proof of \ref{th:carafrechet3}$\Rightarrow$\ref{th:carafrechet1} in Theorem \ref{th:carafrechet} we can replace equation \eqref{equa:frechminisalto} with
\begin{equation}\label{equa:frechminisaltovector} \tag{\ref*{equa:frechminisalto}*}
0 \leq d(x_i,y_i)-(g(x_i)-g(y_i))(z) < D\varepsilon.
\end{equation}
In a similar way, we can replace equation \eqref{equa:frechidesalto} with 
\begin{equation}\label{equa:frechidesaltovector} \tag{\ref*{equa:frechidesalto}*}
0 \leq d(x_{i_t},y_{i_{t+1}})-(g(x_{i_t})-g(y_{i_{t+1}}))(z) < mD\varepsilon .
\end{equation}
Notice now that, since $z$ is a point of Fr\'echet differentiability, by \v Smulyan's lemma we get the existence of a function $\delta:\mathbb R^+\longrightarrow \mathbb R^+$ with $\lim\limits_{\varepsilon\rightarrow 0}\delta(\varepsilon)=0$ and such that the following condition holds
$$\left.\begin{array}{c}
     v^*\in S_{Z^*} \\
     v^*(z)\geq 1-\varepsilon 
\end{array} \right\}\Rightarrow \Vert z^*-v^*\Vert\leq \delta(\varepsilon).$$
The previous condition together with \eqref{equa:frechminisaltovector} and \eqref{equa:frechidesaltovector} implies that
\begin{align*}
\Vert d(x_i,y_i)z^*-(g(x_i)-g(y_i))\Vert &\leq \delta(D\varepsilon) \\
\Vert d(x_{i_t},y_{i_{t+1}})z^*-(g(x_{i_t})-g(y_{i_{t+1}})) \Vert &\leq \delta(m D \varepsilon)
\end{align*}

Continuing the proof in the same way we obtain that $\Vert g-f\otimes z^*\Vert_\infty < \eta(\varepsilon)$ for some function $\eta:\RR^+\longrightarrow\RR^+$ that only depends on $M$ and $n$ and such that $\lim\limits_{\varepsilon\rightarrow 0}\eta(\varepsilon)=0$. From here, we easily deduce by \v Smulyan's lemma that $f\otimes z^*$ is the Fr\'echet derivative of $\mu\otimes z$.

Finally, if we suppose instead that $z$ is a point of G\^ateaux differentiability and fix $\varepsilon=0$ and $\delta\equiv 0$ in the above argument, we deduce in a similar way that $f\otimes z^*$ is the G\^ateaux derivative of $\mu\otimes z$. The following statement sums up our findings:

\begin{theorem}\label{th:diferenvectorval}
Let $M$ be a uniformly discrete and bounded metric space and let $\mu\in S_{\mathcal F(M)}$ be a point of Fr\'echet differentiability with finite support. Let $X$ be a Banach space and $x\in S_X$. Then:
\begin{enumerate}[label={\upshape{(\alph*)}}]
    \item If $x$ is a point of G\^ateaux differentiability of $X$, then $\mu\otimes x$ is a point of G\^ateaux differentiability of $\mathcal F(M)\pten X$.
    \item If $x$ is a point of Fr\'echet differentiability of $X$, then $\mu\otimes x$ is a point of Fr\'echet differentiability of $\mathcal F(M)\pten X$.
\end{enumerate}
\end{theorem}

\begin{remark}\label{remark:contraejetensodiffe}
 Let $X$ and $Y$ be two Banach spaces. Notice that, in general, it is not true that if $x\in S_X$ and $y\in S_Y$ are points of Fr\'echet (resp. G\^ateaux) differentiability then $x\otimes y$ is a point of Fr\'echet (resp. G\^ateaux) differentiability of $X\pten Y$.
\begin{enumerate}
    \item   Given $1<p\leq q<\infty$ it follows from \cite[Example VI.4.1]{hww} and \cite[Theorem 5.33]{rya} that $\ell_p\pten \ell_{q^*}$ is a non-reflexive $L$-summand in its bidual, where $\frac{1}{q}+\frac{1}{q^*}=1$ (see \cite[Chapters I and IV]{hww} for background), and so it does not contain any point of Fr\'echet differentiability \cite[p. 168, (b)]{hww}.
    
    \item If $X=Y=\ell_2$ and $x\in S_X$ then $x\otimes x$ attains its norm at the functionals $T,S\in (X\pten X)^*=\mathcal L(X,X^*)$ defined by
    \begin{align*}
    T(u\otimes v) &= \duality{u,v} , \\
    S(u\otimes v) &= \duality{x,u} \duality{x,v} .
    \end{align*}
    This shows that $x\otimes x$ is not a point of G\^ateaux differentiability in spite of the fact that $x$ is a point of G\^ateaux differentiability. The authors are grateful to Gin\'es L\'opez-P\'erez for pointing out this example to them.
    \end{enumerate}
\end{remark} 

Let us finish this section with two open questions.
First, we have obtained in Theorem \ref{theo:caragateauxarbi} a metric characterisation of those limits of convex series of molecules which are points of G\^ateaux differentiability. However, in the case of Fr\'echet differentiability we have only characterised elements of finite support. 

\begin{question}
Is there any metric characterisation of limits of convex series with infinite support that are points of Fr\'echet differentiability?
\end{question}

Second, all our work on points of Fr\'echet differentiability has focused on limits of convex series of molecules, as this allows us to consider a fixed set of points in the metric space $M$. However, the following question makes sense.

\begin{question}
Let $M$ be a uniformly discrete and bounded metric space. Is there any point of Fr\'echet differentiability $\mu\in\lipfree{M}$ that is not the limit of a convex series of molecules?
\end{question}

\section*{Acknowledgments}

This research was conducted during visits of the first author to the University of Granada in 2019 and 2020, for which he wishes to express his gratitude. We also thank Gin\'es L\'opez-P\'erez, Miguel Mart\'in and Anton\'in Proch\'azka for their useful remarks on the topic of the paper.


\end{document}